\documentclass{ws-jktr}
\usepackage{mathrsfs}
\usepackage{amsfonts}
\usepackage{dsfont}
\usepackage{amssymb,latexsym,amsmath}
\usepackage{graphicx}
\usepackage{float}
\usepackage{caption}
\usepackage{verbatim}
\usepackage{tikz}
\usepackage{hyperref}
\usetikzlibrary{matrix}
\newtheorem{innercustomthm}{Theorem}
\newenvironment{customthm}[1]
  {\renewcommand\theinnercustomthm{#1}\innercustomthm}
  {\endinnercustomthm}

\newcommand\fnote[1]{\captionsetup{font=small}\caption*{#1}}
\newcommand{\lastcfrac}[2]{%
  \vphantom{\cfrac{#1}{#2}}%
  \raisebox{\dimexpr1ex-\height}{%
    $\displaystyle
      \raisebox{.5\height}{$\ddots$}+\cfrac{#1}{#2}
    $%
  }%
}

\title{Slope of Orderable Dehn Filling of Two-Bridge Knots}
\author{Xinghua Gao}
\address{KIAS, 85 Hoegiro Dongdaemun-gu, Seoul 02455, Republic of Korea \\ xgao29@126.com}

\begin{document}

\maketitle

\begin{abstract}
In this paper, we study the Riley polynomial of double twist knots with higher genus. Using the root of the Riley polynomial, we compute the range of rational slope $r$ such that $r$-filling of the knot complement has left-orderable fundamental group. Further more, we make a conjecture about left-orderable surgery slopes of two-bridge knots.
\end{abstract}



\section{Introduction}
The study of Riley polynomial of two-bridge knots dates back to \cite{riley}. In a series of papers by Ryoto Hakamata and Masakazu Teragaito \cite{HT1, HT3}, they studied a special class of two-bridge knot, the double twisted knot $J(k,l)$ (see Figure \ref{diagram}). 
Under their convention, the half-twists in the upper box are left-handed (resp. right-handed) if $k > 0$ (resp. $k < 0$), while those in the lower box are right-handed (resp. left-handed) if $l > 0$ (resp. $l < 0$).  By symmetry, $J(k,l)$ is isotopic to $J(-l,-k)$. For example, $J(3,2)$ is the knot $5_2$, while $J(-3,2)$ is the figure-eight knot $4_1$. We will follow this convention in the paper.

\begin{figure}[H]
\center
\includegraphics[width=50mm]{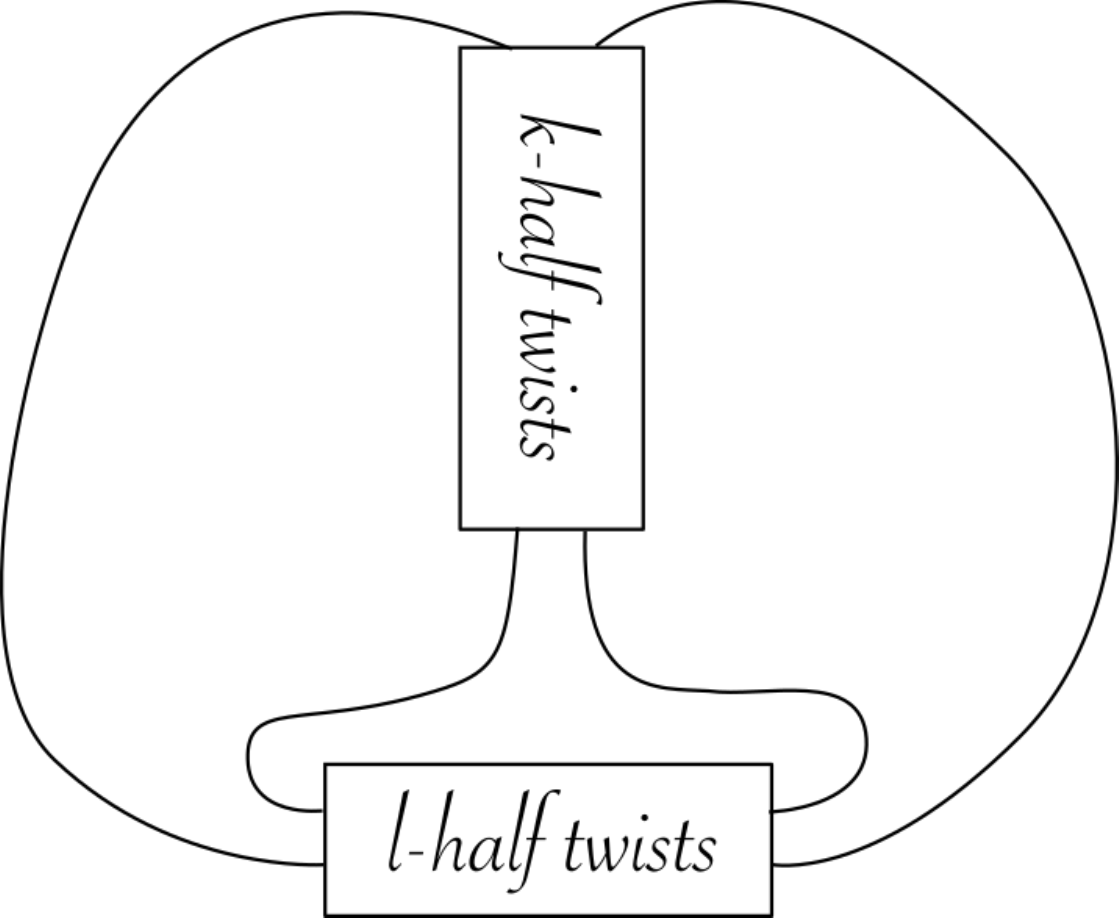}
\quad
\includegraphics[width=55mm]{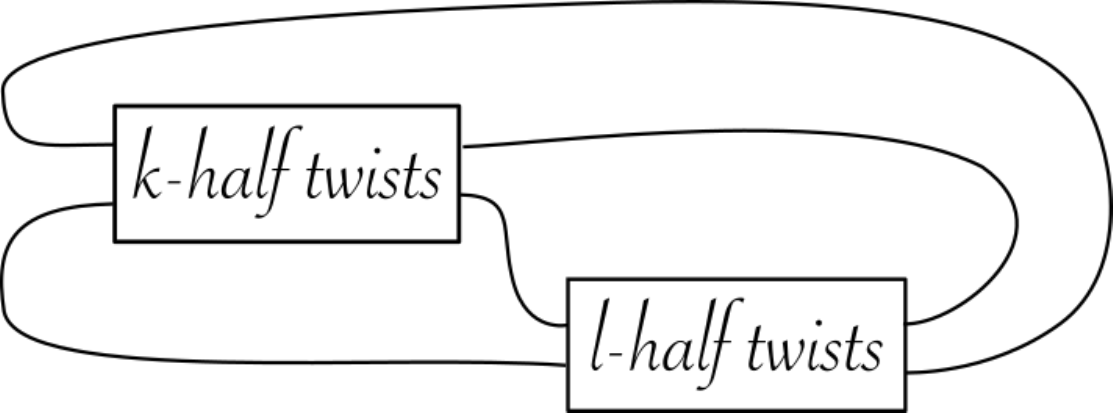}
\caption{The double twist knot $J(k,l)$}
\label{diagram}
\end{figure}

In \cite{HT3}, the special case of the double twisted knot $K=J(k,l)$ with $k=2m$ and $l=2n$ half twists is studied. And a property of the fundamental group of the Dehn filling of $K$ called left-orderability is discussed. A non-trivial group $G$ is called left-orderable if there is a strict total ordering invariant under left multiplication. That is to say, if $f < g$ then $hf < hg$ for any $f, g, h \in G$. We call a 3-manifold orderable if its fundamental group is left-orderable. The main result of \cite{HT3} states:
\begin{theorem}\cite[Theorem 1.1]{HT3}
Let $K = J(2m, 2n)$ be the hyperbolic genus one 2-bridge knot in the 3-sphere $S^3$ as illustrated in Figure 1. Let I be the interval defined by 
\begin{displaymath}
I= 
\begin{cases} 
(-4n, 4m)& \mbox{if } m>0 \text{ and } n>0, \\ 
(4m, -4n)& \mbox{if } m<0 \text{ and } n<0, \\ 
[0, \max\{4m,-4n\})& \mbox{if } m>0 \text{ and } n<0, \\ 
(\min\{4m,-4n\}, 0]& \mbox{if } m<0 \text{ and } n>0. 
\end{cases}
\end{displaymath}
Then any slope in $I$ is left-orderable. That is, $\pi_1(K(r))$ is left-orderable.
\end{theorem}
As Tran informed the author, the above result is also proved independently by him in \cite{Tran_surg} using a different method. And in Tran's recent paper \cite{tran19}, he improved the above result by showing slopes in $(-\infty, 0)$ also give rise to orderable Dehn fillings.

While in the papers \cite{HT3}, the authors mostly focused on the double twist knot with $2m$ and $2n$ half twists, which has genus $1$. In this paper, I computed the Riley polynomial of the double twist knot with $2m+1$ and $2n$ half twists $J(2m+1,2n)$, which has higher genus. By studying the root of the Riley polynomial, I proved the following theorem of double twist knots.

\begin{theorem}\label{main}
For rational $r\in I$, the Dehn filling of $J(2m+1,2n)$ of slope $r$ is orderable, where
\begin{displaymath}
I= 
\begin{cases} 
(-\infty, 1)& \mbox{if } m> 0 \text{ and } n>0, \text{ or } m<-1 \text{ and } n>1,\\
(-1, \infty)& \mbox{if } m< -1 \text{ and } n<0,\text{ or } m>0 \text{ and } n<-1,\\
[0, 4n)& \mbox{if } m<-1 \text{ and } n>0, \\ 
(4n, 0]& \mbox{if } m>0 \text{ and } n<0.
\end{cases}
\end{displaymath}
\end{theorem}
\begin{remark}
When $m=-1, 0$ or $n=0$, $J(-1, 2n)$, $J(1, 2n)$ and $J(2m+1, 0)$ become torus knots. So these three cases are excluded. All two-bridge links except the torus links are hyperbolic (See for example \cite[Corollary 2]{menasco}). Therefore as a subset of two-bridge links, all the double twist knots in our theorem are hyperbolic.
\end{remark}

\begin{remark}
After the author posted this paper on arxiv, Tran informed the author that in version 3 of his paper \cite{tran19}, he independently computed the interval in the first two cases of Theorem \ref{main}. 
\end{remark}

Section 2-7 of this paper are devoted to proving the theorem. In Section \ref{examples}, we use computer program to check the theorem for some examples. In Section \ref{conj}, we make connections with general two-bridge knots.

\section{SL$_2(\mathbb{R})$ Representations of the Fundamental Group of $J(k,2n)$}

We will mostly follow the notations in \cite{HT3}.

Let us look at the link diagram of $J(k, l)$ in Figure \ref{diagram}. 
The link diagram $J(k, l)$ is a knot (called the double twist knot) if and only if either $k$ or $l$ is even. Then one may apply symmetries of the knot to assume that $l = 2n$ is even.  
Moreover, since $J(k,2n)$ and $J(-k,-2n)$ differ by mirror, then in our computation, we can assume $n>0$ without loss of generality. 

The fundamental group of $J(k, 2n)$ has the following presentation $\pi_1(S^3-J(k,2n))=\left<x, y | w^nx=yw^n \right>$, where
\begin{equation}
w= 
\begin{cases} 
(xy^{-1})^m(x^{-1}y)^m, & \mbox{if } k = 2m, \\ 
(xy^{-1})^mxy(x^{-1}y)^m, & \mbox{if } k = 2m+ 1. 
\end{cases}
\end{equation}

In the first case $J(2m,2n)$ corresponds to $K(m,n)$ as in \cite{HT3}. We are interested in the second case $J(2m+1,2n)$, double twist knot with $2m+1$ half twists and $2n$ half twists, which has genus $n$ when $m \neq -1$ (see e.g. \cite[Corollary 8.7.5]{cromwell}). Under the above presentation of the fundamental group, the meridian $\mathcal{M}=x$ and the longitude $\mathcal{L}=w_*^nw^nx^{-4n}$, where $w_*=(yx^{-1})^myx(y^{-1}x)^m$ is obtained from $w$ by reversing the order of letters. See Section 4 of \cite{hoste} for more details about the presentation of $\pi_1(S^3-J(k,2n))$ and $\mathcal{L}$. But remark that there is a slight difference in the notation and choice of framing.

In this paper, we will study the case when $k=2m+1$ in $J(k, l)$. 

Let $\rho$ be a SL$_2(\mathbb{R})$ representation of $\pi_1(S^3-J(2m+1,2n))$. Then by taking conjugation, we can assume that
\begin{equation}
\rho(x)=
\begin{pmatrix} 
\sqrt{t} & 1/\sqrt{t} \\
0 & 1/\sqrt{t} 
\end{pmatrix}
,\quad
\rho(y)=
\begin{pmatrix} 
\sqrt{t} & 0 \\
-s\sqrt{t} & 1/\sqrt{t} 
\end{pmatrix},
\end{equation}
where $s$ and $t$ are real numbers. 
To compute the image of $\mathcal{M}$ and $\mathcal{L}$ under $\rho$, we will need some computational trick.
Let $f_m$ and $g_m$ be two sequences of polynomials in the variable $s$ defined by inductive relations $f_{m+2}-(s+2)f_{m+1}+f_m=0$ and $g_{m+2}-(s+2)g_{m+1}+g_m=0$, with initial conditions $f_0=g_0=1$, $f_1=s+1$ and $g_1=s+2$. 
Set $f_{-m}=f_{m-1}$ for $m\geq 1$. Set $g_{-1}=0$ and $g_{-m}=-g_{m-2}$ for $m\geq 2$. So both $f_m$ and $g_m$ are defined for any $m \in \mathbb{Z}$.

\begin{lemma}\cite[Lemma 2.2]{HT3}
The closed formulas for $f_m$ and $g_m$ are
\begin{displaymath}
f_m=\sum_{i=0}^m\binom{m+i}{m-i}s^i, \quad g_m=\sum_{i=0}^m\binom{m+1+i}{m-i}s^i.
\end{displaymath}
In particular, all coefficients of $f_m$ and $g_m$ are positive integers, and the degree of both $f_m$ and $g_m$ are $m$. Also, $f_m$ and $g_m$ are monic.
\end{lemma}
Then it is easy to check that $f_m$ and $g_m$ satisfy the following properties. 
\begin{lemma}\cite[Lemma 2.3]{HT3}\label{relations}
\begin{itemize}
\item[(1)] $f_m+g_{m-1}=g_m$\\
\item[(2)] $f_m+sg_{m}=f_{m+1}$\\
\item[(3)] $f_m^2=sg_mg_{m-1}+1$\\
\end{itemize}
\end{lemma}

Now we can compute the image of $w$ under $\rho$.
\begin{proposition}\label{W}
Let $W=\rho(w)$, then
\begin{equation}
W=
\begin{pmatrix} 
tf_m^2 - sg_m^2 & \frac{f_m g_m}{t}-f_m g_{m-1} \\
stf_m g_{m-1}-sf_m g_m & \frac{f_m^2}{t} -sg_{m-1}^2
\end{pmatrix}.
\end{equation}
\end{proposition}

\begin{proof}
We compute directly, 
\begin{equation*}
\rho(xy^{-1})=
\begin{pmatrix} 
s+1 & 1 \\
s & 1 
\end{pmatrix},
\end{equation*}

\begin{equation*}
\rho(x^{-1}y)=
\begin{pmatrix} 
s+1 & -1/t \\
-st & 1 
\end{pmatrix},
\end{equation*}

\begin{equation*}
\rho(xy)=
\begin{pmatrix} 
t-s & 1/t \\
-s & 1/t 
\end{pmatrix}.
\end{equation*}
Then
\begin{equation*}
\rho(xy^{-1})^m=
\begin{pmatrix} 
f_m & g_{m-1} \\
f_m-f_{m-1} & g_{m-1}-g_{m-2} 
\end{pmatrix}
=
\begin{pmatrix} 
f_m & g_{m-1} \\
sg_{m-1} & f_{m-1}
\end{pmatrix},
\end{equation*}

\begin{equation*}
\rho(x^{-1}y)^m=
\begin{pmatrix} 
f_m & -\frac{f_{m+1}-(s+1)f_m}{st} \\
-stg_{m -1}& g_{m}-(s+1)g_{m-1} 
\end{pmatrix}
=
\begin{pmatrix} 
f_m & -\frac{g_{m-1}}{t} \\
-stg_{m-1} & f_{m-1} 
\end{pmatrix}.
\end{equation*}
So 
\begin{displaymath}
W=\rho(w)=\rho(xy^{-1})^m\rho(xy)\rho(x^{-1}y)^m=
\begin{pmatrix} 
tf_m^2-sg_m^2 & \frac{f_m g_m}{t}-f_m g_{m-1} \\
stf_m g_{m-1}-sf_m g_m & \frac{f_m^2}{t}-sg_{m-1}^2
\end{pmatrix}.
\end{displaymath}
\end{proof}
Denote
\begin{displaymath}
W=
\begin{pmatrix} 
w_{1,1} & w_{1,2} \\
w_{2,1} & w_{2,2} 
\end{pmatrix}.
\end{displaymath}

Let $z$ and $1/z$ be the eigenvalues of $W$. Set $\tau=\text{tr} W=z+1/z$ and $\tau_k=z^{k-1}+z^{k-3}+\cdots +z^{3-k}+z^{1-k}$. Then $\tau_k$ satisfies $\tau_{k+2}-\tau\tau_{k+1}+\tau_k=0$.
\begin{lemma}\cite[Lemma 3.5]{HT3}
For $W^n=\begin{pmatrix} z_{1,1} & z_{1,2} \\z_{2,1} & z_{2,2} \end{pmatrix}=\rho(w^n)$, 
we have 
\begin{displaymath}
W^n=
\begin{pmatrix} 
w_{1,1}\tau_n-\tau_{n-1} &  w_{1,2}\tau_n\\
w_{2,1}\tau_n & \tau_{n+1}-w_{1,1}\tau_n 
\end{pmatrix}.
\end{displaymath}
\end{lemma}
The proof of \cite[Lemma 3.5]{HT3} does not depend on the explicit expression of $W$. So this lemma still holds in our case.

\section{Longitude}
In this section, we follow the procedures in Section 5 of \cite{HT3} to compute the image of the longitude $\mathcal{L}$ under certain representation $\rho_s$, which will be defined below.

Conjugating $\rho$ by $Q= \begin{pmatrix} t-1 &  1\\0 & \sqrt{t}-1/\sqrt{t} \end{pmatrix}$, 
we can diagonalize $\rho(x)$. Call this new representation $\rho_s$. So
\begin{displaymath}
\rho_s(x)= 
\begin{pmatrix} 
\sqrt{t} &  0\\
0 & \frac{1}{ \sqrt{t} }
\end{pmatrix}, \text{ and }
\rho_s(y)= 
\begin{pmatrix} 
\frac{t-s-1}{ \sqrt{t}-\frac{1}{\sqrt{t}} } &  \frac{s}{ (\sqrt{t}-\frac{1}{\sqrt{t}}  )^2} -1\\
-s & \frac{s+1-1/t}{ \sqrt{t}-\frac{1}{\sqrt{t}}  }
\end{pmatrix}.
\end{displaymath}
We will compute $\rho_s(\mathcal{L})$.
Let 
\begin{displaymath}
\rho_s(w^n)=
\begin{pmatrix} 
v_{1,1} & v_{2,1} \\
v_{1,2} & v_{2,2} 
\end{pmatrix},
\text{ and } \sigma=\frac{s(\sqrt{t}-\frac{1}{\sqrt{t}})^2}{(\sqrt{t}-\frac{1}{\sqrt{t}})^2-s}.
\end{displaymath}
 We will first compute $\rho_s(w_*^n)$.

\begin{lemma}
\begin{displaymath}
\rho_s(w_*^n)=
\begin{pmatrix} 
v_{1,1} & \frac{v_{2,1}}{\sigma} \\
v_{1,2}\sigma & v_{2,2} 
\end{pmatrix}.
\end{displaymath}
\end{lemma}

\begin{proof}
It's easy to compute that
\begin{displaymath}
\rho_s(xy)=
\begin{pmatrix} 
\frac{t^2-st-t}{t-1} & \frac{s\sqrt{t}}{(t-1/t)^2} -\sqrt{t}\\
-\frac{s}{\sqrt{t}} & \frac{st+t-1}{t(t-1)}
\end{pmatrix}=
\begin{pmatrix} 
\frac{t^2-st-t}{t-1} & -\frac{s\sqrt{t}}{\sigma} \\
-\frac{s}{\sqrt{t}} & \frac{st+t-1}{t(t-1)}
\end{pmatrix},
\end{displaymath}

\begin{displaymath}
\rho_s(yx)=
\begin{pmatrix} 
\frac{t^2-st-t}{t-1} & \frac{s}{\sqrt{t}(t-1/t)^2} -\frac{1}{\sqrt{t}}\\
-s\sqrt{t} & \frac{st+t-1}{t(t-1)}
\end{pmatrix}=
\begin{pmatrix} 
\frac{t^2-st-t}{t-1} & -\frac{s}{\sqrt{t}\sigma} \\
-s\sqrt{t} & \frac{st+t-1}{t(t-1)}
\end{pmatrix}.
\end{displaymath}
We can observe that the $(1,2)$-th entry of $\rho_s(yx)$ is exactly the $(2,1)$-th entry of $\rho_s(xy)$ divided by $\sigma$, and the $(2,1)$-th entry of $\rho_s(yx)$ is exactly the $(1,2)$-th entry of $\rho_s(xy)$ multiplied by $\sigma$.  And the other two entries of $\rho_s(yx)$ and $\rho_s(xy)$ coincide.

The authors of \cite{HT3} proved in Lemma 5.1 that similar relation holds between $\rho_s(y^{-1}x)$ and $\rho_s(xy^{-1})$, and also between $\rho_s(yx^{-1})$ and $\rho_s(x^{-1}y)$. 
Moreover, they showed that this relation is preserved under matrix multiplication;
\begin{displaymath}
\begin{pmatrix} 
a & b \\
c & d
\end{pmatrix}
\begin{pmatrix} 
p & q \\
r & s
\end{pmatrix}
=
\begin{pmatrix} 
ap+br & aq+bs \\
cp+dr & cq+ds
\end{pmatrix},
\end{displaymath}

\begin{displaymath}
\begin{pmatrix} 
p & \frac{r}{\sigma} \\
q\sigma & s
\end{pmatrix}
\begin{pmatrix} 
a & \frac{c}{\sigma} \\
b\sigma & d
\end{pmatrix}
=\begin{pmatrix} 
ap+br & \frac{cp+dr}{\sigma} \\
(aq+bs)\sigma & cq+ds
\end{pmatrix}.
\end{displaymath}
Therefore the same relation must also hold for $\rho_s(w^n)$ and $\rho_s(w_*^n)$.
\end{proof}

\begin{lemma}\cite[Lemma 5.4]{HT3}\label{uw}
For $U=\begin{pmatrix} u_{1,1} & u_{1,2} \\ u_{2,1} & u_{2,2}\end{pmatrix}=\rho_s(w)$, 
\begin{align*}
&u_{1,1}=w_{1,1}+\frac{w_{2,1}}{t-1}, & &u_{1,2}=\sqrt{t}\left(w_{1,2}-\frac{w_{1,1}}{t-1}\right)+\frac{\sqrt{t}}{t-1}\left(w_{2,2}-\frac{w_{2,1}}{t-1}\right), \\
&u_{2,1}=\frac{w_{2,1}}{\sqrt{t}}, & &u_{2,2}=w_{2,2}-\frac{w_{2,1}}{t-1}.
\end{align*}
\end{lemma}

\begin{lemma}\cite[Lemma 5.5]{HT3}\label{uv}
For $U^n=\begin{pmatrix} v_{1,1} & v_{2,1} \\v_{1,2} & v_{2,2} \end{pmatrix}=\rho_s(w)^n$, 
we have
\begin{align*}
&v_{1,1}=u_{1,1}\tau_n-\tau_{n-1}, & &v_{1,2}=u_{1,2}\tau_n, \\
&v_{2,1}=u_{2,1}\tau_n, & &v_{2,2}=\tau_{n+1}-u_{1,1}\tau_n.
\end{align*}
\end{lemma}
Again, the above two lemmas still hold because no explicit expression of entries of $W$ is involved.

Let $B_s$ be the $(1,1)$-entry of $\rho_s(\mathcal{L})$. In our case, $B_s$ obtains the form in the following lemma, which is different from \cite[Lemma 5.6]{HT3}.

\begin{lemma}\label{Bs}
$\displaystyle B_s=-\frac{u_{2,1}}{u_{1,2}\sigma}t^{-2n}$.
\end{lemma}

\begin{proof}
We compute $\rho_s(\mathcal{L})=\rho_s(w_*^nw^nx^{-4n})$ first. 
\begin{align*}
\rho_s(x)^{-4n}&=Q\rho(x)^{-4n}Q^{-1}
=Q\begin{pmatrix} 
1/t & -1-1/t \\
0 & t
\end{pmatrix}^{2n}Q^{-1}\\
&=Q\begin{pmatrix} 
1/t^{2n} & * \\
0 & t^{2n}
\end{pmatrix}^{2n}Q^{-1}
=\begin{pmatrix} 
1/t^{2n} & * \\
0 & t^{2n}
\end{pmatrix}.
\end{align*}
Then
\begin{align*}
\rho_s(\mathcal{L})=\rho_s(w_*^n)\rho_s(w^n)\rho_s(x)^{-4n}&=
\begin{pmatrix} 
v_{1,1} & \frac{v_{2,1}}{\sigma} \\
v_{1,2}\sigma & v_{2,2} 
\end{pmatrix}
\begin{pmatrix} 
v_{1,1} & v_{1,2} \\
v_{2,1} & v_{2,2} 
\end{pmatrix}
\begin{pmatrix} 
1/t^{2n} & * \\
0 & t^{2n}
\end{pmatrix}\\
&=\begin{pmatrix} 
(v_{1,1}^2+ \frac{v_{2,1}^2}{\sigma}) t^{-2n} & *\\
(v_{1,1}v_{1,2}\sigma+ v_{2,1}v_{2,2})  t^{-2n}& *
\end{pmatrix}.
\end{align*}

So $\displaystyle B_s=(v_{1,1}^2+ \frac{v_{2,1}^2}{\sigma}) t^{-2n}$. To simplify $B_s$, we need to use the relation det$U^n=v_{1,1}v_{2,2}-v_{1,2}v_{2,1}=1$. Moreover $v_{1,1}v_{1,2}\sigma+v_{2,1}v_{2,2}=0$, because $\rho_s(\mathcal{L})$ is diagonal. Then

\begin{align*}
B_s&=(v_{1,1}^2+ \frac{v_{2,1}^2}{\sigma}) t^{-2n}=(-\frac{v_{1,1}v_{2,2}}{v_{1,2}} + v_{2,1})  \frac{v_{2,1}}{\sigma} t^{-2n}\\
&=(-\frac{v_{1,2}v_{2,1}+1}{v_{1,2}} + v_{2,1})  \frac{v_{2,1}}{\sigma} t^{-2n}= -\frac{v_{2,1}}{v_{1,2}\sigma} t^{-2n}\\
&\stackrel{Lemma\ \ref{uv}}{=\joinrel=\joinrel=} -\frac{u_{2,1}\tau_n}{u_{1,2}\tau_n\sigma} t^{-2n}= -\frac{u_{2,1}}{u_{1,2}\sigma} t^{-2n}.
\end{align*}

\end{proof}

\begin{proposition}\label{Bs2}
Let $B_s$ be the $(1,1)$-entry of $\rho_s(\mathcal{L})$, where $\mathcal{L}$ is the longitude of $J(2m+1,2n)$. Then
\begin{equation*}
B_s=\frac{g_m-tg_{m-1}}{g_{m-1}-tg_m}t^{-2n}.
\end{equation*}
\end{proposition}

\begin{proof}
We compute directly using Lemma \ref{uw} and Proposition \ref{W}. 
\begin{equation*}
u_{2,1}=\frac{w_{2,1}}{\sqrt{t}}=\frac{s}{\sqrt{t}}f_m(tg_{m-1}-g_m), 
\end{equation*}

\begin{align*}
&u_{1,2}=\sqrt{t}\left(w_{1,2}-\frac{w_{1,1}}{t-1}\right)+\frac{\sqrt{t}}{t-1}\left(w_{2,2}-\frac{w_{2,1}}{t-1}\right)\\
&=\frac{\sqrt{t}}{t(t-1)} \left((t-1)f_mg_m-(t^2-t)f_mg_{m-1}-t^2f_m^2+stg_m^2\right)\\
&+\frac{\sqrt{t}}{t(t-1)} \left(-stg_{m-1}^2+f_m^2-\frac{st^2}{t-1}f_mg_{m-1}+\frac{st}{t-1}f_mg_m\right)\\
&=\frac{\sqrt{t}}{t(t-1)}\left( (t-1+\frac{st}{t-1})f_mg_m-(t^2-t+\frac{st^2}{t-1})f_mg_{m-1}+(1-t^2)f_m^2+st(g_m^2-g_{m-1}^2) \right)\\
&=\frac{\sqrt{t}}{t(t-1)}f_m\left( (t-1+\frac{st}{t-1})g_m-(t^2-t+\frac{st^2}{t-1})g_{m-1}+(1-t^2)(g_m-g_{m-1})+st(g_m+g_{m-1}) \right)\\
&=\frac{\sqrt{t}}{t(t-1)}f_m\left( (t+\frac{st}{t-1}-t^2+st)g_m-(-t+\frac{st^2}{t-1}+1-st)g_{m-1} \right)\\
&=\frac{\sqrt{t}(st-(t-1)^2)}{t(t-1)^2}f_m\left( tg_m-g_{m-1} \right)=-\frac{s}{\sqrt{t}\sigma}f_m\left( tg_m-g_{m-1} \right).
\end{align*}
Applying Lemma \ref{Bs}, we have
\begin{equation*}
B_s=-\frac{u_{2,1}}{u_{1,2}\sigma}=\frac{g_m-tg_{m-1}}{g_{m-1}-tg_m}t^{-2n}.
\end{equation*}

\end{proof}

\section{Root of the Riley Polynomial}
The Riley polynomial of two-bridge knot was first studied by Robert Riley in his paper \cite{riley}. The root of the Riley polynomial of a two-bridge knot $K$ describes all the nonabelian PSL$_2(\mathbb{C})$ representations of $\pi_1(S^3-K)$. By examining real roots of the Riley polynomial of $K$, we will be able to construct PSL$_2(\mathbb{R})$ representations that we need.
\begin{proposition}
The Riley polynomial of $K=J(2m+1,2n)$ is 
\begin{equation}\label{rileyp}
\phi_K(s,t)=(\tau_{n+1}-\tau_{n})+(t+1/t-s-2)f_mg_{m-1}\tau_n,
\end{equation}
where $\tau=(t+1/t-s-2)f_m^2+2$.
\end{proposition}

\begin{proof}
By \cite[Theorem 1]{riley}, the Riley polynomial for a two-bridge knot $K$ is $\phi_K(s,t)=z_{1,1}+(1-t)z_{1,2}$. So for the double twist knot $J(2m+1,2n)$, 
\begin{align*}
\phi_K(s,t)
&=(w_{1,1}\tau_n-\tau_{n-1})+(1-t)w_{1,2}\tau_n\\
&=(\tau_{n+1}-\tau_{n})+((1-t)w_{1,2}+1-w_{2,2})\tau_n.
\end{align*}
\begin{align*}
(1-t)w_{1,2}+1-w_{2,2}&=(1-t)(\frac{f_m g_m}{t}-f_m g_{m-1})+f_m^2-sg_mg_{m-1}-(\frac{f_m^2}{t} -sg_{m-1}^2)\\
&=(t+1/t-s-2)f_mg_{m-1}\tau_n.
\end{align*}
So $\phi_K(s,t)=(\tau_{n+1}-\tau_{n})+(t+1/t-s-2)f_mg_{m-1}\tau_n$.

Direct computation using Proposition \ref{W} and Lemma \ref{relations} shows 
\begin{equation}\label{tau}
\tau=\text{tr} W=tf_m^2 - sg_m^2+ \frac{f_m^2}{t} -sg_{m-1}^2=(t+1/t-s-2)f_m^2+2.
\end{equation}
\end{proof}

\begin{lemma}
When $m>0$, both $f_m$ and $g_m$ have $m$ simple negative roots. 
\end{lemma}

\begin{proof}
First, we show both $f_m$ and $g_m$ have at least one sign change when $m\neq 0$. Notice that both $f_m$ and $g_m$ are monic polynomials in terms of $s$ of degree $m$. In particular $g_m(2x-2)$ is the Chebyshev Polynomial of the second kind. 
Therefore $g_m$ has exactly $m$ different negative simple roots $2\cos\left(\frac{k\pi}{m+1}\right)-2$, $k=1,2,\cdots m$. 

When $s\rightarrow \infty$, both $f_m$ and $g_m$ approach $+\infty$. So $g_m<0$ on $(2\cos \left(\frac{2j\pi}{m+1}\right)-2, 2\cos\left(\frac{(2j-1)\pi}{m+1}\right)-2)$, while $g_m>0$ on $(2\cos \left(\frac{(2j+1)\pi}{m+1}\right)-2, 2\cos\left(\frac{2j\pi}{m+1}\right)-2)$ and $(2\cos\left(\frac{\pi}{m+1}\right)-2,\infty)$, $j=1, \ldots, \lfloor\frac{m}{2}\rfloor$. 
Since $g_{m-1}\left(2\cos\left(\frac{(2j-1)\pi}{m}\right)-2\right)=0$ and $2\cos\left(\frac{(2j-1)\pi}{m}\right)-2\in \left(2\cos\left(\frac{2j\pi}{m+1}\right)-2,2\cos\left(\frac{(2j-1)\pi}{m+1}\right)-2\right)$ for $j=1, \ldots, \lfloor\frac{m}{2}\rfloor$, then $f_{m}\left(2\cos\left(\frac{(2j-1)\pi}{m}\right)-2\right)<0$  by Lemma \ref{relations} (1). 
Similarly, we have $2\cos\left(\frac{2j\pi}{m}\right)-2\in \left(2\cos\left(\frac{(2j+1)\pi}{m+1}\right)-2,2\cos\left(\frac{2j\pi}{m+1}\right)-2\right)$ and $g_{m-1}\left(2\cos\left(\frac{2j\pi}{m}\right)-2\right)=0$ for $j=1, \ldots, \lceil\frac{m}{2}\rceil-1$. Then it follows from Lemma \ref{relations} (1), that $f_{m}\left(2\cos\left(\frac{2j\pi}{m}\right)-2\right)>0$. And moreover $f_m(0)=1>0$. Then the degree $m$ polynomial $f_m$ has exactly $m$ sign changes, which implies that $f_m$ should have at exactly $m$ simple roots. 
\end{proof}

Let $r_{f_m}<0$ be the largest root of $f_m$, i.e. $r_{f_m}$ has the smallest absolute value among all roots of $f_m$. Then for $m>0$, both $g_m$ and $g_{m-1}$ are positive on $(r_{f_m}, 0)$.  
Setting $t=1$ in the Riley polynomial $\phi_K(s,t)=(\tau_{n+1}-\tau_{n})+(t+1/t-s-2)f_mg_{m-1}\tau_n=0$, we get $\phi_K(s,1)=(\tau_{n+1}-\tau_{n})-sf_mg_{m-1}\tau_n=0$, with $\tau=-sf_m^2+2$. Then we can choose $s_0>0$ to be the smallest positive solution of $\phi_K(s,1)=0$ with respect to $s$. 

\begin{proposition}\label{riley_sol}
Set $T=t+\frac{1}{t}$. 
\begin{itemize}
\item[(1)] When $n>1$ or $n=1$ $m>0$, the Riley polynomial (\ref{rileyp}) has a root $\frac{-C_2}{f_m^2}<T-s-2<\frac{-C_1}{f_m^2}$ for $s>0$, where $C_1$ and $C_2$ are positive constants. Moreover, $0<T<2$ for $s\in (0, s_0)$.  \\ 
\item[(2)] When $m<-1$ and $n>0$, the Riley polynomial has a root $\frac{C_3}{f_m}<T-s-2<\frac{C_4}{f_m}$ with $r_{f_m}<s<0$, where $C_3$ and $C_4$ are positive constants. Moreover, $T>2$ for $s\in (r_{f_m}, 0)$.
\end{itemize}
\end{proposition}

\begin{proof}

(1) Suppose $n>1$, we set $\tau_{n+1}=\tau_{n}$ as in \cite[Lemma 4.1]{HT3}. Then we can find two roots $z=e^{\frac{\pi}{2n+1}i}$ and $z=e^{\frac{3\pi}{2n+1}i}$, with
\begin{displaymath}
\tau_n(e^{\frac{\pi}{2n+1}i})=\tau_{n+1}(e^{\frac{\pi}{2n+1}i})>0, \quad \tau_n(e^{\frac{3\pi}{2n+1}i})=\tau_{n+1}(e^{\frac{3\pi}{2n+1}i})<0.
\end{displaymath}
Let $C_1=2-2\cos(\frac{\pi}{2n+1})>0$, $C_2=2-2\cos(\frac{3\pi}{2n+1})>0$. Then
\begin{align*}
\phi_K(s,T=s+2-\frac{C_1}{f_m^2})=-\frac{C_1g_{m-1}}{f_m}\tau(e^{\frac{\pi}{2n+1}i})<0, \\
\phi_K(s,T=s+2-\frac{C_2}{f_m^2})=-\frac{C_2g_{m-1}}{f_m}\tau(e^{\frac{3\pi}{2n+1}i})>0.
\end{align*}
Since $\phi_K(s,T)$ is a polynomial function of $T$, it is continuous. So it has a root $\displaystyle \frac{-C_2}{f_m^2}<T-s-2<\frac{-C_1}{f_m^2}$ by Intermediate Value Theorem, with $C_1,C_2 \in (0,4)$. 

When $m>0$ and $n=1$, $\phi_K(s,t)=(\tau-1)+(t+1/t-s-2)f_mg_{m-1}$, where $\tau=(t+1/t-s-2)f_m^2+2$. So $\phi_K(s,t)=(T-s-2)f_mg_{m}+1$, which has the solution $\displaystyle T=s+2+\frac{-1}{f_mg_m}$. Since $\displaystyle \frac{f_m}{g_m}$ is continuous on $[0,s_0]$, then it must be bounded below by $C_1$ and above by $C_2$. And it is not hard to see that $0<C_1<C_2<1$, as $g_m>f_m \geq 1$ when $s\geq 0$. Therefore when $n=1$, we still have $\displaystyle \frac{-C_2}{f_m^2}<T-s-2<\frac{-C_1}{f_m^2}$. 

To see $0<T<2$ for $s\in (0, s_0)$, notice that by assumption $s_0$ is the smallest positive value for $s$ such that $T=2$.

(2) When $m<-2$, let $m'=-m$, then $f_m=f_{m'-1}$ and $g_{m-1}=-g_{m'-1}$. So the Riley polynomial becomes $\tau_{n+1}-\tau_{n}-(T-s-2)f_{m'-1}g_{m'-1}\tau_{n}=0$. 
Since $2\cos\left(\frac{\pi}{m'-1}\right)-2$, the largest root of $g_{m'-2}$, is smaller than $r_{f_m}$, then $g_{m'-2}$ is positive and increasing on $s\in (r_{f_m},0)$. As a result, $g_{m'-1}=f_{m'-1}+g_{m'-2}$ must also be positive and increasing on $(r_{f_m},0)$. Setting $s=r_{f_m}$ and $f_{m'-1}=0$ in Lemma \ref{relations} (1) and (3), we get $g_{m'-1}(0)=\sqrt{\frac{1}{-r_{f_m}}}$. So $g_{m'-1}\in(\sqrt{\frac{1}{-r_{f_m}}},m')$ and $f_{m'-1}\in (0,1)$, when $s\in (r_{f_m},0)$. Since $-1<r_{f_m}<0$, $\sqrt{\frac{1}{-r_{f_m}}}>1$, claim that we can choose $C_3, C_4>0$ such that $2+C_3f_{m'-1}<\tau<2+C_4f_{m'-1}$. 
In fact, we can choose $C_4>0$ large enough such that $1+C_4\sqrt{\frac{1}{-r_{f_m}}}>2+C_4>2+C_4f_{m'-1}$. When $\tau>2$, both $\tau_n$ and $\tau_{n+1}$ are positive. So 
\begin{align*}
\phi_K(s,\tau=2+C_4f_{m'-1})=\tau_{n+1}-\tau_{n}-C_4g_{m'-1}\tau_{n}&<\tau_{n+1}-(1+C_4\sqrt{\frac{1}{-r_{f_m}}})\tau_{n}\\
 &<\tau_{n+1}-\tau\tau_{n}=-\tau_{n-1}<0.
\end{align*}
When $m=-2$, $m'=2$. So $f_{m'-1}=s+1$, $g_{m'-1}=s+2$. Let $C_4\geq 1$, then
\begin{align*}
\phi_K(s,\tau=2+C_4f_{m'-1})=&\phi_K(s, \tau=2+C_4(s+1))=\tau_{n+1}-\tau_{n}-C_4(s+2)\tau_{n}\\
=&\tau_{n+1}-(2+C_4(s+1)+C_4-1)\tau_{n}\\
=&\tau_{n+1}-\tau\tau_{n}+(1-C_4)\tau_{n}\\
=&-\tau_{n-1}+(1-C_4)\tau_{n} \leq -\tau_{n-1}<0.
\end{align*}

For $\tau >2$, it is easy to verify that $\frac{\tau_{n+1}}{\tau_n}$ is an increasing function (for example using first derivative test), so $\frac{\tau_{n+1}}{\tau_n}>\frac{n+1}{n}$. So we can choose $C_3>0$ such that $1+C_3m'<\frac{n+1}{n}$ and it follows that
\begin{displaymath}
\phi_K(s,\tau=2+C_3f_{m'-1})=\tau_{n+1}-\tau_{n}-C_3g_{m'-1}\tau_{n}>\tau_{n+1}-(1+C_3m')\tau_{n}>0.
\end{displaymath}

Apparently $T>2$ when $s\rightarrow r_{f_m}+$. To see $T=t+1/t>2$ for all $s\in (r_{f_m}, 0)$, we will need to look at $B_s$. In the proof of Lemma \ref{Bs}, notice that $\displaystyle B_s=(v_{1,1}^2+ \frac{v_{2,1}^2}{\sigma}) t^{-2n}$. So $B_s$ is a polynomial in terms of $s, \frac{1}{s}, t, \frac{1}{t}$ and $\frac{1}{t-1}$. Since by Proposition \ref{Bs2}  $B_s=-1$ when $t$ approaches and takes the value $1$, then $t=1$ is a removable singularity of $B_s$.  Therefore $B_s$ is continuous for $s\in(r_{f_m}, 0)$. Now assume $T=2$ when $s$ equals some $s_1\in (r_{f_m}, 0)$ and choose $s_1$ such that $T>2$ for $s \in (r_{f_m}, s_1)$. In fact, we can assume $t>1$ for $s \in (r_{f_m}, s_1)$, because similar computations with $1/t>1$ instead of $0<t<1$ could be carried out. Then $g_m-tg_{m-1}=tg_{m'-1}-g_{m'-2}=(t-1)g_{m'-1}+g_{m'-1}-g_{m'-2}=(t-1)g_{m'-1}+f_{m'-1}>0$ for $s \in (r_{f_m}, s_1)$. Moreover, $g_{m-1}-tg_m=tg_{m'-2}-g_{m'-1}>(t-1)g_{m'-2}>0$. This implies $\displaystyle B_s=\frac{g_m-tg_{m-1}}{g_{m-1}-tg_m}t^{-2n}>0$ for $s \in (r_{f_m}, s_1)$, and it follows by continuity that $B_s\geq 0$ for $s \in (r_{f_m}, s_1]$, which contradicts $B_s(s=s_1)=-1$.

\end{proof}

\section{Slopes}
In this section, we will compute slopes of asymptotes of the graph of the root of the Riley polynomial under logarithmic scale.
\begin{lemma}
\begin{itemize}
\item[(1)]
$\displaystyle \lim_{s\rightarrow r_{f_m}+} t=\infty$ when $m< -1$ and $n>0$. \\
\end{itemize}
\end{lemma}

\begin{proof}
When $m<-1$ and $n>0$, by Proposition \ref{riley_sol} (2) we have $s+2+\frac{C_3}{f_m}<T<s+2+\frac{C_4}{f_m}$. So $\displaystyle \lim_{s\rightarrow r_{f_m}+} T=\infty$, and we can choose $t$ such that $\displaystyle \lim_{s\rightarrow r_{f_m}+} t=\infty$.
\end{proof}

\begin{lemma} \label{limit_B}
\ 
\begin{itemize}
\item[(1)] $\displaystyle \lim_{s\rightarrow r_{f_m}+} B_st^{2n}=1$ when $m< -1$.\\
\item[(2)] $\displaystyle \lim_{s\rightarrow 0+} B_s=1$ when $m\neq -1$.\\
\item[(3)] $\displaystyle \lim_{s\rightarrow s_0-} B_s=-1$ and $\displaystyle \lim_{s\rightarrow s_0-} \text{Ln}(B_s)=-\pi+2d\pi$ for some $d\in \mathbb{Z}$. \\
\end{itemize}
\end{lemma}

\begin{proof}
(1) When $f_m=0$, from Lemma \ref{relations}(1) we know $g_m=g_{m-1}$. So 
\begin{displaymath}
\lim_{s\rightarrow r_{f_m}+}B_st^{2n}=\lim_{s\rightarrow r_{f_m}+}\frac{g_m-tg_{m-1}}{g_{m-1}-tg_m}=\lim_{s\rightarrow r_{f_m}+}\frac{g_m(1-t)}{g_m(1-t)}=1.
\end{displaymath}

(2) To prove the second limit, set $s=0$ in $\tau$, then $\tau=t+1/t$. When $m\neq -1$, set $s=0$ in $\phi_K(s,t)=0$, then we have $(\tau_{n+1}-\tau_{n})+m(t+1/t-2)\tau_n=0$. So $t^{n+1}-t^{-n-1}-t^{n}+t^{-n}+m(t+1/t-2)(t^{n}-t^{-n})=0$, which simplifies to $t^{2n}[m-(m+1)t]-[(m+1)-mt]=0$. So
\begin{displaymath}
\lim_{s\rightarrow 0+} B_s=\frac{m+1-mt}{m-(m+1)t}t^{-2n}=1.
\end{displaymath}

(3) When $s=s_0$, we have $t=1$ and 
\begin{displaymath}
\lim_{s\rightarrow s_0-} B_s=\frac{g_m-g_{m-1}}{g_{m-1}-g_m}=-1.
\end{displaymath}
So $\displaystyle \lim_{s\rightarrow s_0-} \text{Ln}(B_s)=-\pi+2d\pi$ for some integer $d$.
\end{proof}
It is hard in general to give a formula to compute $d$, but we know $|2d-1|\geq 1$, and we can choose proper framing such that $2d-1>0$. 

Let $A_s$ be the $(1,1)$-entry of $\rho_s(\mathcal{M})=\rho_s(x)$, which equals $\sqrt{t}$. We define $g:(,)\rightarrow \mathbb{R}$ to be $\displaystyle g(s)=-\frac{\ln (B_s)}{\ln(A_s)}$ and examine the image of $g$.

\begin{proposition}\label{slope}
For points corresponding to the root of the Riley polynomial, the image of $g$ contains every number in $I$, with 
\begin{displaymath}
I=
\begin{cases} 
(-\infty, 0)& \mbox{if } m\neq -1 \text{ and } n>1, \text{ or } m>0 \text{ and } n=1, \\
(0, 4n)& \mbox{if } m<-1 \text{ and } n>0, \\ 
\end{cases}
\end{displaymath}
\end{proposition}
\begin{proof}
\begin{displaymath}
\lim_{s\rightarrow s_0-} g(s)=-\lim_{s\rightarrow s_0-}\frac{\ln (B_s)}{\ln (A_s)}=-\frac{\ln(-1)}{\ln(1^+)}=-\frac{\pi}{0^+}=-\infty.
\end{displaymath}

Applying Lemma \ref{limit_B}, when $m\neq -1$ we have
\begin{displaymath}
\lim_{s\rightarrow 0+} g(s)=-\lim_{s\rightarrow 0+}\frac{\ln (B_s)}{\ln (A_s)}=-\lim_{s\rightarrow 0+}\frac{2\ln (B_s)}{\ln (t)}=0.
\end{displaymath}
When $m<-1$ and $n>0$, 
\begin{displaymath}
\lim_{s\rightarrow r_{f+}} g(s)=-\lim_{s\rightarrow r_{f+}}\frac{\ln (B_s)}{\ln (A_s)}=-\lim_{s\rightarrow r_{f+}}\frac{2\left(\ln (B_st^{2n})-2n\ln (t)\right)}{\ln (t)}=4n.
\end{displaymath}
Then the lemma follows from Proposition \ref{riley_sol}.
\end{proof}

\section{Root of the Alexander Polynomial}\label{sec_alex}
Setting $s=0$ in (\ref{tau}), then $\tau=T=t+1/t$ and $\phi_K(s=0,t)=\tau_{n+1}-\tau_n+m(\tau-2)\tau_n=0$, which simplifies to 
\begin{equation}\label{riley0}
\begin{split}
0&=\phi_K(0,t)=(m+1)(t^n+t^{-n})+(2m+1)\sum_{i=0}^{n-1}(t^i+t^{-i})(-1)^{n-i}.\\
\end{split}
\end{equation}

Next, we compute the Alexander polynomial $\Delta_{J(2m+1,2n)}(a)$ of $J(2m+1,2n)$ and show that it is the same as (\ref{riley0}). 

Let $\Delta_{T(p,q)}$ be the Alexander polynomial of the torus knot $T(p,q)$.
\begin{lemma} (see for example \cite[Example 9.15]{knots_BZ})
\begin{displaymath}
\Delta_{T(p,q)}(a)=\frac{(a^{pq}-1)(a-1)}{(a^p-1)(a^q-1)}\times a^{-\frac{(p-1)(q-1)}{2}}.
\end{displaymath}
\end{lemma}
The extra factor $a^{-\frac{(p-1)(q-1)}{2}}$ is multiplied to normalize $\Delta_{T(p,q)}(a)$ so that it is symmetric, i.e. $\Delta_{T(p,q)}(a^{-1})=\Delta_{T(p,q)}(a)$.

\begin{figure}[H]
\center
\includegraphics[width=130mm]{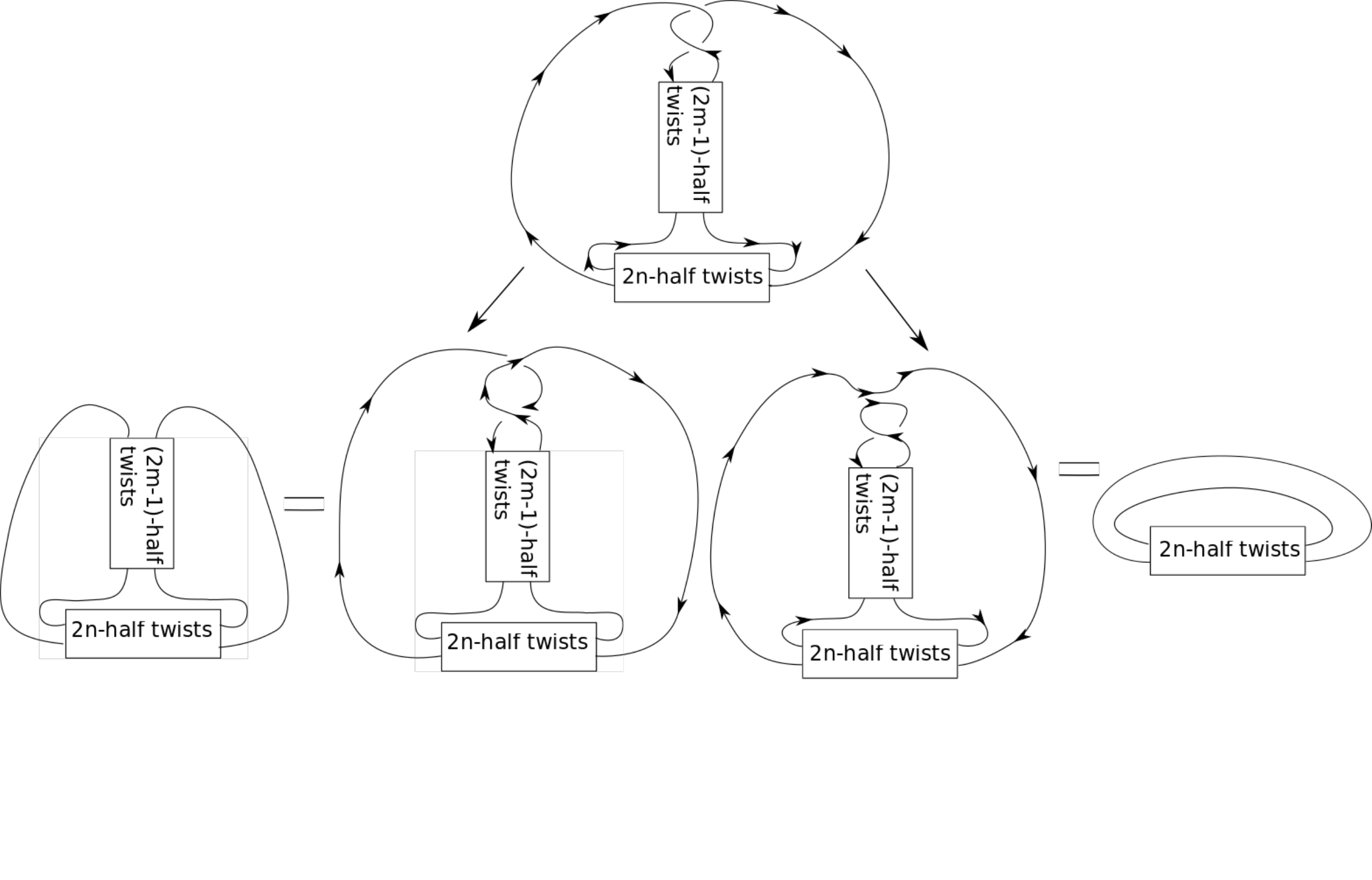}
\caption{The Skein relations}
\label{skein_diagram}
\end{figure}

\begin{proposition}
The Alexander polynomial of the double twist knot $J(2m+1,2n)$ is $\Delta_{J(2m+1,2n)}(a)=(m+1)(a^n+a^{-n})+(2m+1)\sum_{i=0}^{n-1}(a^i+a^{-i})(-1)^{n-i}$. 
\end{proposition}

\begin{proof}

We will use the fact that the Alexander polynomial satisfies the Skein relations (see for example \cite[Chapter 6]{adams}) and prove by induction.

Let $\Delta_{J(2m+1,2n)}(a)$ be the Alexander polynomial of the double twist knot $J(2m+1,2n)$. Normalize $\Delta_{J(2m+1,2n)}(a)$ so that $\Delta_{J(2m+1,2n)}(a)=\Delta_{J(2m+1,2n)}(a^{-1})$ and the coefficient of its highest degree term is positive. From the Skein relations as shown in Figure  \ref{skein_diagram}, we have
\begin{align*}
\Delta_{J(2m+1,2n)}-\Delta_{J(2m-1,2n)}&=(\sqrt{t}-1/\sqrt{t})\Delta_{J(0,2n)}, \\
\Delta_{J(2m-1,2n)}-\Delta_{J(2m-3,2n)}&=(\sqrt{t}-1/\sqrt{t})\Delta_{J(0,2n)}, \\
&\cdots \\
\Delta_{J(3,2n)}-\Delta_{J(1,2n)}&=(\sqrt{t}-1/\sqrt{t})\Delta_{J(0,2n)}.
\end{align*}
Adding the above equations all together, we have $\Delta_{J(2m+1,2n)}-\Delta_{J(1,2n)}=m(\sqrt{t}-1/\sqrt{t})\Delta_{J(0,2n)}$. 
Notice that $J(0,2n)$ is the torus link $T(2,2n)$, and $J(1,2n)$ is the torus knot $T(2,2n+1)$.  Let $\Delta_{T(p,q)}$ be the Alexander polynomial of the torus knot or link $T(p,q)$.  
Then $\Delta_{J(0,2n)}= \Delta_{T(2,2n)}$ and $\Delta_{J(1,2n)}= \Delta_{T(2,2n+1)}$. 
Applying the Skein relations of torus links, we have $(\sqrt{t}-1/\sqrt{t})\Delta_{J(0,2n)}=(\sqrt{t}-1/\sqrt{t})\Delta_{T(2,2n)}=\Delta_{T(2,2n+1)}-\Delta_{T(2,2n-1)}$. So $\Delta_{J(2m+1,2n)}-\Delta_{T(2,2n+1)}=\Delta_{J(2m+1,2n)}-\Delta_{J(1,2n)}=m(\Delta_{T(2,2n+1)}-\Delta_{T(2,2n-1)})$. Thus 
\begin{align*}
\Delta_{J(2m+1,2n)}
&=(m+1)\Delta_{T(2,2n+1)}-m\Delta_{T(2,2n-1)}\\
&=(m+1)\frac{(a^{2(2n+1)}-1)(a-1)}{(a^{2n+1}-1)(a^2-1)a^n}-m\frac{(a^{2(2n-1)}-1)(a-1)}{(a^{2n-1}-1)(a^2-1)a^{n-1}}\\
&=(m+1)\frac{(a^{2n+1}+1)}{(a+1)a^n}-m\frac{a^{2n-1}+1}{(a+1)a^{n-1}}\\
&=(m+1)\sum_{i=-n}^{n}(a^i)(-1)^{n+i} - m\sum_{i=-n+1}^{n-1}(a^i)(-1)^{n-1+i}\\
&=(m+1)(a^n+a^{-n})+(2m+1)\sum_{i=0}^{n-1}(a^i+a^{-i})(-1)^{n-i}.
\end{align*}

\end{proof}

We can see that the Alexander polynomial of $J(2m+1,2n)$ is exactly the same as its Riley polynomial when $s=0$ (as shown in (\ref{riley0})). So when $s=0$, $t$ takes the value of a root $\xi$ of the Alexander polynomial.

\begin{proposition}
\begin{itemize}
\item[(1)] When $m\neq -1$ and $n>1$ or $m>0$ and $n=1$, $\displaystyle \lim_{s\rightarrow 0+} t=\xi$ is a unit complex root of  the Alexander polynomial $\Delta_{J(2m+1,2n)}$.\\
\item[(2)] When $m< -1$ and $n>0$, $\displaystyle \lim_{s\rightarrow 0-} t=\xi$ is a positive real root of $\Delta_{J(2m+1,2n)}$.
\end{itemize}
\end{proposition}

\begin{proof}
As we see in (\ref{riley0}), setting $s=0$ in the Riley polynomial, we get 
\[ \phi_K(0,t)=(m+1)(t^n+t^{-n})+(2m+1)\sum_{i=0}^{n-1}(t^i+t^{-i})(-1)^{n-i}, \]
which is the same as the Alexander polynomial $\Delta_{J(2m+1,2n)}$. 
Notice that it is a palindrome (symmetric) polynomial of even degree $2n$ after multiplying by $t^n$.
So by \cite[Theorem 1]{palindrome}, it has a complex root on the unit circle when $m\neq -1$ and $n>1$. 
When $n=1$, $\phi_K(0,t)=(m+1)t+(2m+1)+(m+1)t^{-1}$. If $m>0$, then $m+1>\frac{1}{2}(2m+1)$. So again by \cite[Theorem 1]{palindrome}, $\phi_K(0,t)$ has a unit complex root.

When $m<-1$, there is a positive real root different from $1$. To see this, let $m'=-m>1$, then $\phi_K(0,t)=t^{2n}(m'-(m'-1)t)+m't-(m'-1)=0$. Set $h(t)=t^{2n}(m'-(m'-1)t)+m't-(m'-1)$. When $t=0$, $h(t)=1-m'<0$; when $t=1$, $h(t)=2>0$. So by Intermediate Value Theorem, $h(t)$ must have a real root between $0$ and $1$, and also a root$>1$ by symmetry of $h(t)$.

This proposition could also be proved by taking $s=0$ in $0<2+\frac{-C_1}{f_m^2}<T<2+\frac{-C_2}{f_m^2}<2$, and in $2<2+\frac{C_3}{f_m^2}<T<2+\frac{C_4}{f_m^2}$ when $m<-1$.
\end{proof}

\section{Left-Orderability}
Boyer, Rolfsen and Wiest proved the following theorem about left-orderability.
\begin{theorem}{\cite[Theorem 1.1]{BRW}}
Suppose that $M$ is a compact, connected and $P^2$-irreducible 3-manifold. A necessary and sufficient condition that $\pi_1(M)$ be left-orderable is that either $\pi_1(M)$ is trivial or there exists a non-trivial homomorphism from $\pi_1(M)$ to a left-orderable group.
\end{theorem}

Consider the Lie group $SU(1, 1) = \left\{
\begin{pmatrix}
\alpha & \beta\\
\overline{\beta} & \overline{\alpha} \\
\end{pmatrix}
|\ |\alpha|^2 - |\beta|^2 = 1 \right\}$.
So we can parameterize $SU(1, 1)$ by $(\gamma, \omega)$ where $\gamma = -\overline{\beta}/\alpha\in \mathbb{C}$ and $\omega = \text{arg} \alpha$ is defined modulo $2\pi$. Then SL$_2\mathbb{R}\simeq \text{SU}(1,1)$ can be described as $\{(\gamma, \omega)\ |\ |\gamma|<1, -\pi \leq \omega<\pi\}$. 
The nonlinear Lie group $\widetilde{\text{PSL}_2\mathbb{R}}$ is defined to be the universal cover of SL$_2\mathbb{R}$ and $\text{PSL}_2\mathbb{R}$. In particular, it can be described as $\{(\gamma, \omega)\in \mathbb{C}\times \mathbb{R}\ |\ |\gamma|<1, -\infty<\omega<\infty\}$ with group operation given by:

\begin{equation*}
\begin{split}
&(\gamma, \omega)(\gamma', \omega')= \\ &\left((\gamma+\gamma'e^{-2i\omega})(1+\bar{\gamma}\gamma'e^{-2i\omega})^{-1}, \omega+\omega'+\frac{1}{2i}\ln{(1+\bar{\gamma}\gamma'e^{-2i\omega})(1+\gamma\bar{\gamma}'e^{2i\omega})^{-1}}
\right).
\end{split}
\end{equation*}
As a subgroup of Homeo$^+(\mathbb{R})$, $\widetilde{\text{PSL}_2\mathbb{R}}$ is left-orderable. We follow the notation in \cite{CD16}. Denote PSL$_2\mathbb{R}$ by $G$, and $\widetilde{\text{PSL}_2\mathbb{R}}$ by $\widetilde{G}$. 
We call an element $\widetilde{g}$ of $\widetilde{G}$ elliptic, parabolic or hyperbolic if it covers an element $g$ of the corresponding type in PSL$_2\mathbb{R}$. In particular, if $\widetilde{g}$ covers $\pm I$, then $\widetilde{g}$ is called central.

Suppose $M$ is a knot complement in a rational homology 3-sphere. Let $R_{\widetilde{G}}(M)=\text{Hom}(\pi_1(M), \widetilde{G})$ be the variety of $\widetilde{G}$ representations of $\pi_1(M)$. Similarly define $R_{\widetilde{G}}(\partial M)=\text{Hom}(\pi_1(\partial M), \widetilde{G})$. For a precise definition of the representation variety, see for example \cite{CS}.
We call a $\widetilde{G}$ representation $\widetilde{\rho}\in R_{\widetilde{G}}(\partial M)$ elliptic, parabolic, hyperbolic or central if $\widetilde{\rho}(\pi_1(\partial M))$ contains the corresponding elements.

So in the case of double twist knot $J(2m+1,2n)$, all the $\widetilde{G}$ representations corresponding to the root of the Riley polynomial described in Proposition \ref{riley_sol} (1) are elliptic, parabolic or central when restricted to the boundary.  All the $\widetilde{G}$ representations corresponding to the solution described in Proposition \ref{riley_sol} (2) are hyperbolic, parabolic or central when restricted to the boundary. 

\subsection{Translation Extension Locus}\cite[Section 4]{CD16}

The name translation extension locus comes from the fact that we need to use translation number in the definition.
For an elements $\widetilde{g}$ in $\widetilde{G}$, define the translation number of $\widetilde{g}$ to be
\begin{displaymath}
 \text{trans}(\widetilde{g})=\lim_{n\to\infty}\frac{\widetilde{g}^n(x)-x}{n} \text{ for some } x\in \mathbb{R}.
\end{displaymath}
Then trans: $R_{\widetilde{G}}(\partial M)\rightarrow H^1(\partial M; \mathbb{R})$ can be defined by taking $\widetilde{\rho}$ to trans$\circ \widetilde{\rho}$.

Let $M$ be a knot complement in a rational homology 3-sphere. To study $\widetilde{G}$ representations of $M$ whose restrictions to $\pi_1(\partial M)$ are elliptic, Culler and Dunfield gave the following definition of translation extension locus.

\begin{definition} \cite[Section 4]{CD16} 
Let $PE_{\widetilde{G}}(M)$ be the subset of representations in $R_{\widetilde{G}}(M)$ whose restriction to $\pi_1(\partial M)$ are either elliptic, parabolic, or central. Consider composition
\[
PE_{\widetilde{G}}(M)\subset R_{\widetilde{G}}(M)\stackrel{\iota^*}{\longrightarrow} R_{\widetilde{G}}(\partial M)\stackrel{\text{trans}}{\longrightarrow} H^1(\partial M; \mathbb{R}).
\]
The closure in $H^1(\partial M; \mathbb{R})$ of the image of $PE_{\widetilde{G}}(M)$ under $\text{trans}\circ\iota^*$ is called translation extension locus and denoted $EL_{\widetilde{G}}(M)$.
\end{definition}
In particular, $EL_{\widetilde{G}}(M)$ contains the $x$-axis, which corresponds to abelian $G$ representations of $\pi_1(M)$ that are elliptic, parabolic, or central when restricted to $\pi_1(\partial M)$. 

Let $D_{\infty}(M)$ be the infinite dihedral group $\mathbb{Z}\rtimes \mathbb{Z}/2\mathbb{Z}$. They showed that the translation extension locus $EL_{\widetilde{G}}(M)$ satisfies the following properties.

\begin{theorem}\cite[Theorem 4.3]{CD16}\label{trans_thm}
The extension locus $EL_{\widetilde{G}}(M)$ is a locally finite union of analytic arcs and isolated points. It is invariant under $D_{\infty}(M)$ with quotient homeomorphic to a finite graph. The quotient contains finitely many points which are ideal or parabolic in the sense defined above. The locus $EL_{\widetilde{G}}(M)$ contains the horizontal axis $L_{0}$, which comes from representations to $\widetilde{G}$ with abelian image.
\end{theorem}
In the notation of this paper, a $\widetilde{G}$ representation corresponds to point with coordinates $(\frac{1}{2\pi}Ln(t), \frac{1}{\pi}Ln(B_s))$ in $EL_{\widetilde{G}}(M)$. Once we build up the translation extension locus, we will use the following lemma to prove left-orderability.

\begin{lemma}\cite[Lemma 4.4]{CD16}\label{trans_lemma}
Suppose $M$ is a compact orientable irreducible $3$-manifold with $\partial M$ a torus, and assume the Dehn filling $M(r)$ is irreducible. If $L_r$ meets EL$_{\widetilde{G}}(M)$ at a nonzero point which is not parabolic or ideal, then $M(r)$ is orderable.
\end{lemma}

\subsection{Holonomy Extension Locus}
In \cite{gao2}, I constructed the holonomy extension locus which is an analog of the translation extension locus and has similar properties to translations extension locus. 

Let $M$ be the complement of a knot in a rational homology 3-sphere.

\begin{definition}\cite[Definition 3.3]{gao2},
Let $PH_{\widetilde{G}}(M)$ be the subset of representations whose restriction to $\pi_1(\partial M)$ are either hyperbolic, parabolic, or central. Consider the composition
\begin{displaymath}
PH_{\widetilde{G}}(M)\subset R^{\text{aug}}_{\widetilde{G}}(M)\stackrel{\iota^*}{\longrightarrow} R^{\text{aug}}_{\widetilde{G}}(\partial M)\stackrel{\text{EV}}{\longrightarrow} H^1(\partial M; \mathbb{R})\times H^1(\partial M; \mathbb{Z})
\end{displaymath}
The closure of $\text{EV}\circ \iota^* (PH_{\widetilde{G}}(M))$ in $H^1(\partial M; \mathbb{R})$ is called the holonomy extension locus and denoted $HL_{\widetilde{G}}(M)$.
\end{definition}

In particular, $HL_{\widetilde{G}}(M)$ contains the $x$-axis, which corresponds to abelian $G$ representations of $\pi_1(M)$ that are hyperbolic, parabolic, or central when restricted to $\pi_1(\partial M)$. 

The holonomy extension locus $HL_{\widetilde{G}}(M)$ satisfies the following properties.

\begin{theorem}\cite[Theorem 3.1]{gao2} 
The holonomy extension locus $HL_{\widetilde{G}}(M)=\bigsqcup_{i,j\in\mathbb{Z}}H_{i,j}(M)$, $-\frac{k_M}{k}\leq j\leq \frac{k_M}{k}$ is a locally finite union of analytic arcs and isolated points. It is invariant under the affine group $D_{\infty}(M)$ with quotient homeomorphic to a finite graph with finitely many points removed. Each component $H_{i, j}(M)$ contains at most one parabolic point and has finitely many ideal points locally.

The locus $H_{0,0}(M)$ contains the horizontal axis $L_0$, which comes from representations to $\widetilde{G}$ with abelian image.
\end{theorem}
Under the notation of this paper, a $\widetilde{G}$ representation corresponds to a point which obtains coordinates $(\frac{1}{2}\ln(t), \ln (B_s), 
\text{trans}(\widetilde{\rho_s}(\mathcal{M})),  \text{trans}(\widetilde{\rho_s}(\mathcal{L})))$ in $HL_{\widetilde{G}}(M)$, or equivalently $(\frac{1}{2}\ln(t), \ln (B_s))\in H_{i,j}(M)$, with $i=\text{trans}(\widetilde{\rho_s}(\mathcal{M}))$ and $j= \text{trans}(\widetilde{\rho_s}(\mathcal{L}))$. We are using $\ln$ instead of Ln in the coordinates now because both $t$ and $B_s$ are real numbers under this setting. 
Once we build up the holonomy extension locus, we will use the following lemma to prove left-orderability.

\begin{lemma}\label{hol_lemma}
If $L_r$ intersects $H_{0,0}(M)$ component of $HL_{\widetilde{G}}(M)$ at non parabolic or ideal points, and assume $M(r)$ is irreducible, then $\pi_1(M(r))$ is left-orderable.
\end{lemma}

\begin{remark}
A point in $EL_{\widetilde{G}}(M)$ (or $HL_{\widetilde{G}}(M)$) is called an ideal point if it does not come from an actual $\widetilde{G}$ representation of $\pi_1(M)$ but only lives in the closure. But in this paper, since we build translation extension locus/holonomy extension locus from the root of the Riley polynomial, every pair  $(s,t)$ corresponds to a $\widetilde{G}$ representation of $\pi_1(S^3-J(2m+1,2n))$. So we do not encounter ideal points.
\end{remark}

\subsection{Proof of the Main Theorem}

Now we can prove our main theorem.
\begin{customthm}{\ref{main}}
For any rational $r\in I$, the Dehn filling of $J(2m+1,2n)$ of slope $r$ is orderable, where
\begin{displaymath}
I= 
\begin{cases} 
(-\infty, 1)& \mbox{if } m> 0 \text{ and } n>0, \text{ or } m<-1 \text{ and } n>1,\\
(-1, \infty)& \mbox{if } m< -1 \text{ and } n<0,\text{ or } m>0 \text{ and } n<-1,\\
[0, 4n)& \mbox{if } m<-1 \text{ and } n>0, \\ 
(4n, 0]& \mbox{if } m>0 \text{ and } n<0.
\end{cases}
\end{displaymath}
\end{customthm}

\begin{proof}
First of all, notice that Dehn filling of $J(2m+1,2n)$ of slope $r$ is irreducible as long as $n\neq 0$ and $m\neq 0,-1$ by \cite[Theorem 2 (a)]{HT_bridge}.
In particular, the $0$ Dehn filling of $J(2m+1,2n)$ is irreducible and has first betti number equal to $1$. Therefore by \cite[Corollary 3.4]{BRW}, $\pi_1(S^3-J(2m+1,2n)(0))$ is left-orderable.

\textbf{Case 1: $m>0$ and $n>0$ or $m<-1$ and $n>1$.} 

From Proposition \ref{riley_sol} (1) and Lemma \ref{slope}, we know there is an arc in $EL_{\widetilde{G}}(J(2m+1,2n))$ going from a point $(\frac{1}{2\pi}\ln\xi, 0)$ on the positive half of the $x$-axis to a point $(0,2d-1)$ (not included in the arc) on the positive half of the $y$-axis, where $\xi$ is a unit complex root of the Alexander polynomial. 
Here $d$ is defined by $\displaystyle \lim_{s\rightarrow s_0-} \text{Ln}(B_s)=-\pi+2d\pi$. 
By the discussion in \ref{slope}, we can choose proper framing such that $2d-1>0$. So by Lemma \ref{trans_lemma}, Dehn filling of rational slope $r\in (-\infty, 0]$ is orderable.  By symmetry of translation extension locus as described in Theorem \ref{trans_thm}, there is also an arc going from $(1-\frac{1}{2\pi}\ln\xi, 0)$ to $(1,-(2d-1))$ (excluded from the arc). Therefore Dehn filling of rational slope $r\in [0, 2d-1)$ is also orderable. So in total, we know the interval of left-orderable Dehn fillings should at least contain $(-\infty,1)$.

\textbf{Case 2: $m<-1$ and $n>0$.}

When $s=0$, the representation $\rho_s$ is reducible, so trans$(\widetilde{\rho_s}(\mathcal{L}))=0$. By continuity of translation number, all representations corresponding to the same continuous root of the Riley polynomial (i.e. for all $s\in (r_{f_m},0)$) should satisfy trans$(\widetilde{\rho_s}(\mathcal{L}))=0$.
From Proposition \ref{riley_sol} (2) and Lemma \ref{slope}, we know there is an arc in $H_{0,0}(J(2m+1,2n))$ going from a point $(\frac{1}{2}\ln\xi, 0)$ on the positive half of the $x$-axis to infinity with asymptote of slope $-4n$, where $\xi$ is a positive real root of the Alexander polynomial.
Applying Lemma \ref{hol_lemma}, then we see immediately that Dehn filling of rational slope $r\in [0, 4n)$ is orderable.

When $n<0$, all computations could be carried out similarly. 
\end{proof}

\section{Examples}\label{examples}


In this section, I will demonstrate the translation extension locus and holonomy extension locus of some double twist knots and some general two-bridge knots. All the figures are produced by the program PE \cite{pe} written by Marc Culler and Nathan Dunfield.
\subsection{Double Twist Knots}
Let us first look at some double twist knots.
Our first example (figure \ref{e94}) is $J(5,4)=9_4$, with $m=2$ and $n=2$.
\begin{figure}
\center
\includegraphics[width=100mm]{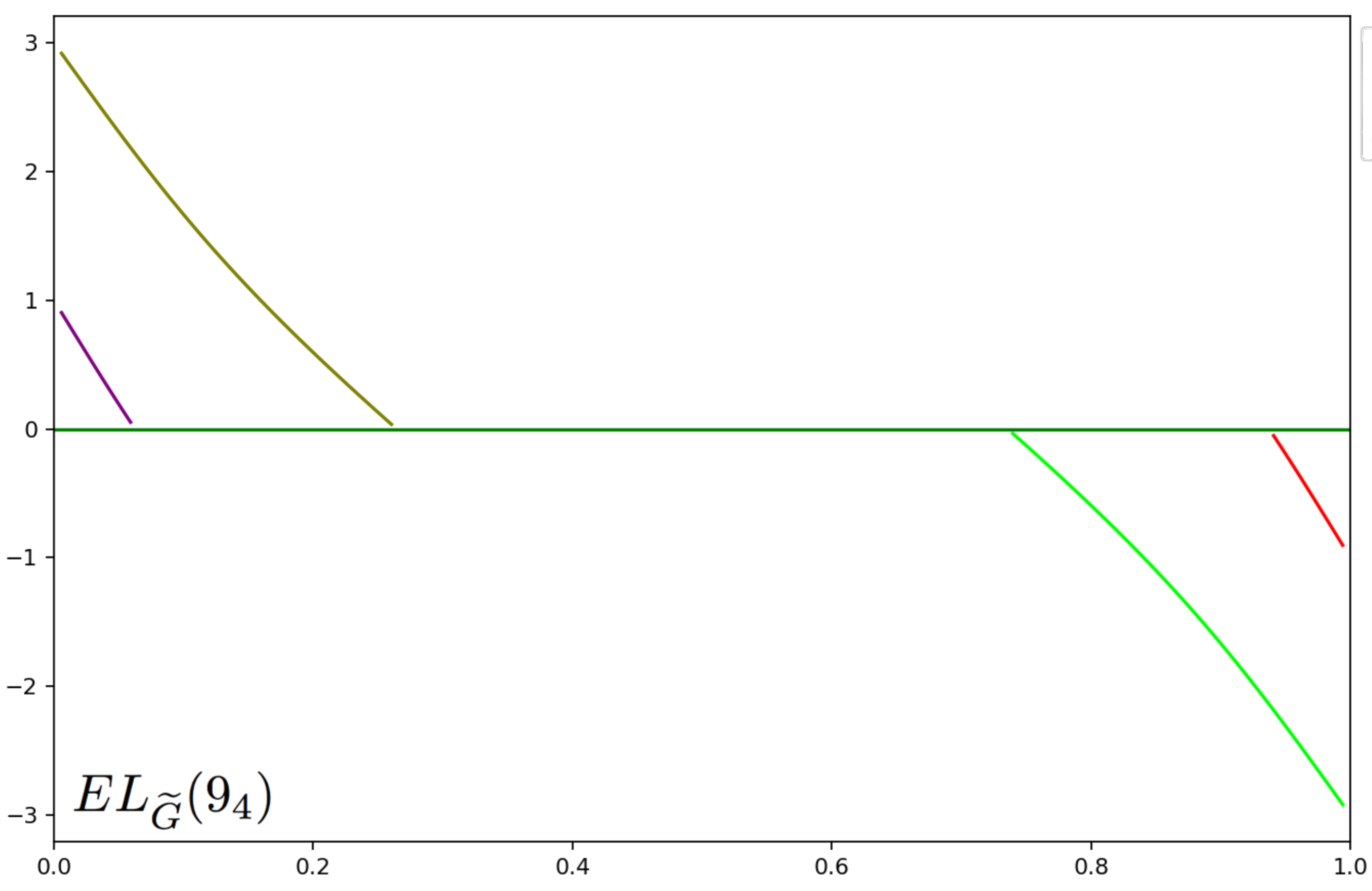}
\caption{Translation Extension Locus $EL_{\widetilde{G}}(9_4)$}
\label{e94}
\fnote{ This figure demonstrates the quotient of translation extension locus of $9_4$ under the action of integral translation in the $x$ direction. The $x$ coordinate of a point in $EL_{\widetilde{G}}(9_4)$ is $\frac{1}{2\pi}Ln(t)$, the $y$ coordinate is $\frac{1}{\pi}Ln(B_s)$. From this figure we know that Dehn filling of $J(5,4)$ of rational slope from $-\infty$ to $3$ is orderable. The number $3$ comes from the translation number of the parabolic element $\widetilde{\rho_{s_0}}(\mathcal{L})$ (or equivalently $\frac{1}{\pi}\text{Ln}(B_{s_0})$). It is in general difficult to compute this number. But we know it has to be at least $1$.}
\end{figure}
Our second example (figure \ref{h62}) is $J(-3,4)=6_2$, with $m=-2$ and $n=2$.
\begin{figure}[H]
\center
\includegraphics[width=100mm]{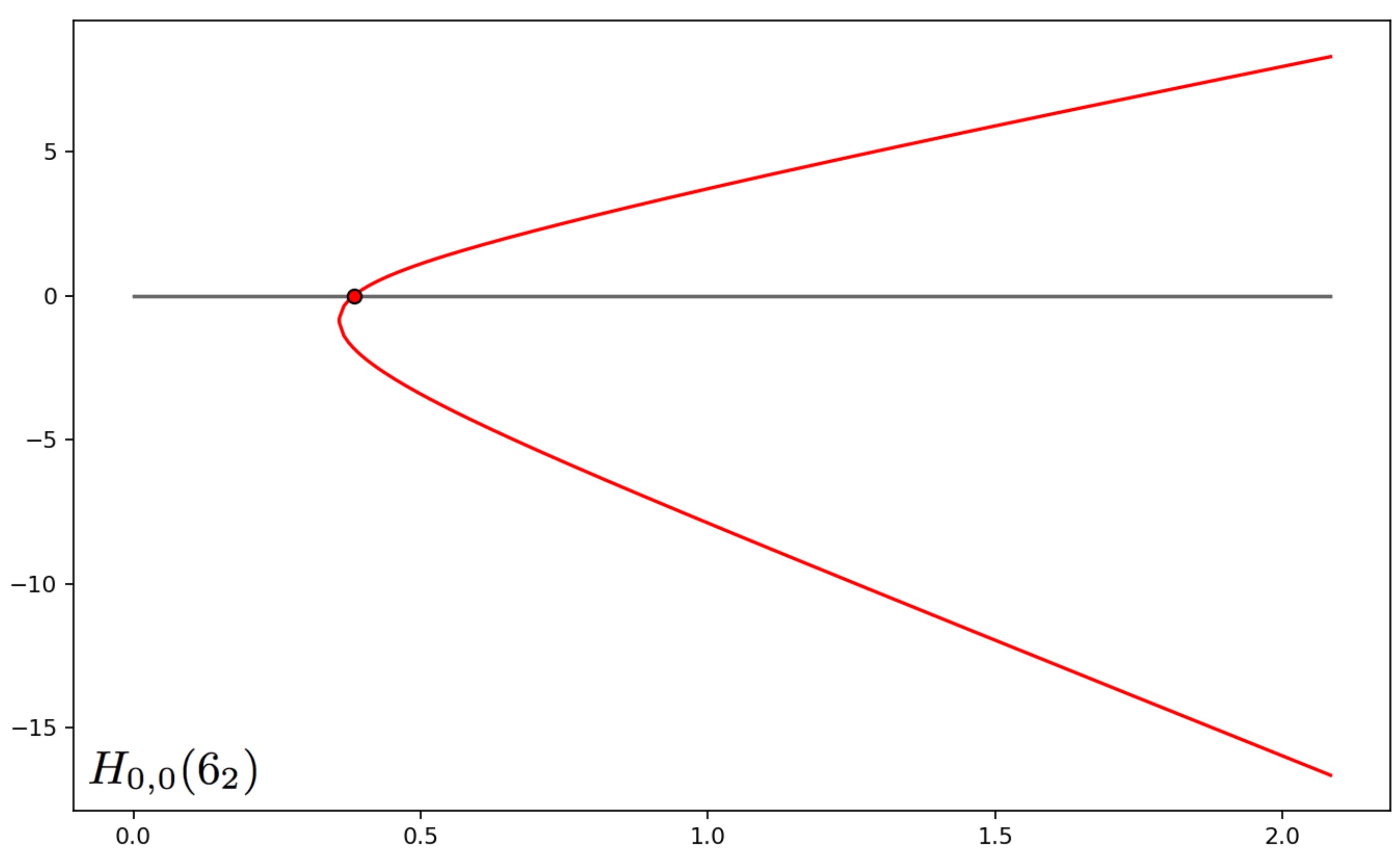}
\caption{Holonomy Extension Locus $HL_{\widetilde{G}}(6_2)$. }
\label{h62}
\fnote{This figure is the quotient of the $H_{0,0}(6_2)$ component of the holonomy extension locus of $6_2$ under the $\mathbb{Z}/2\mathbb{Z}$ action of reflection about the origin. The $x$ coordinate of a point in $H_{0,0}(6_2)$ is $\frac{1}{2}\ln(t)$, the $y$ coordinate is $\ln(B_s)$. The asymptotes have slope $-8$ and $4$. So we know Dehn filling of $J(-3,4)$ of rational slope $r\in(-8, 4)$ is orderable. Unfortunately, our Theorem \ref{main} could only predict the slope $-8=-4n$, but not the other slope $4$. }
\end{figure}

\subsection{Two-Bridge Knots}
Now let us look at more general examples of two-bridge knots.

\begin{figure}[H]
\center
\includegraphics[width=95mm]{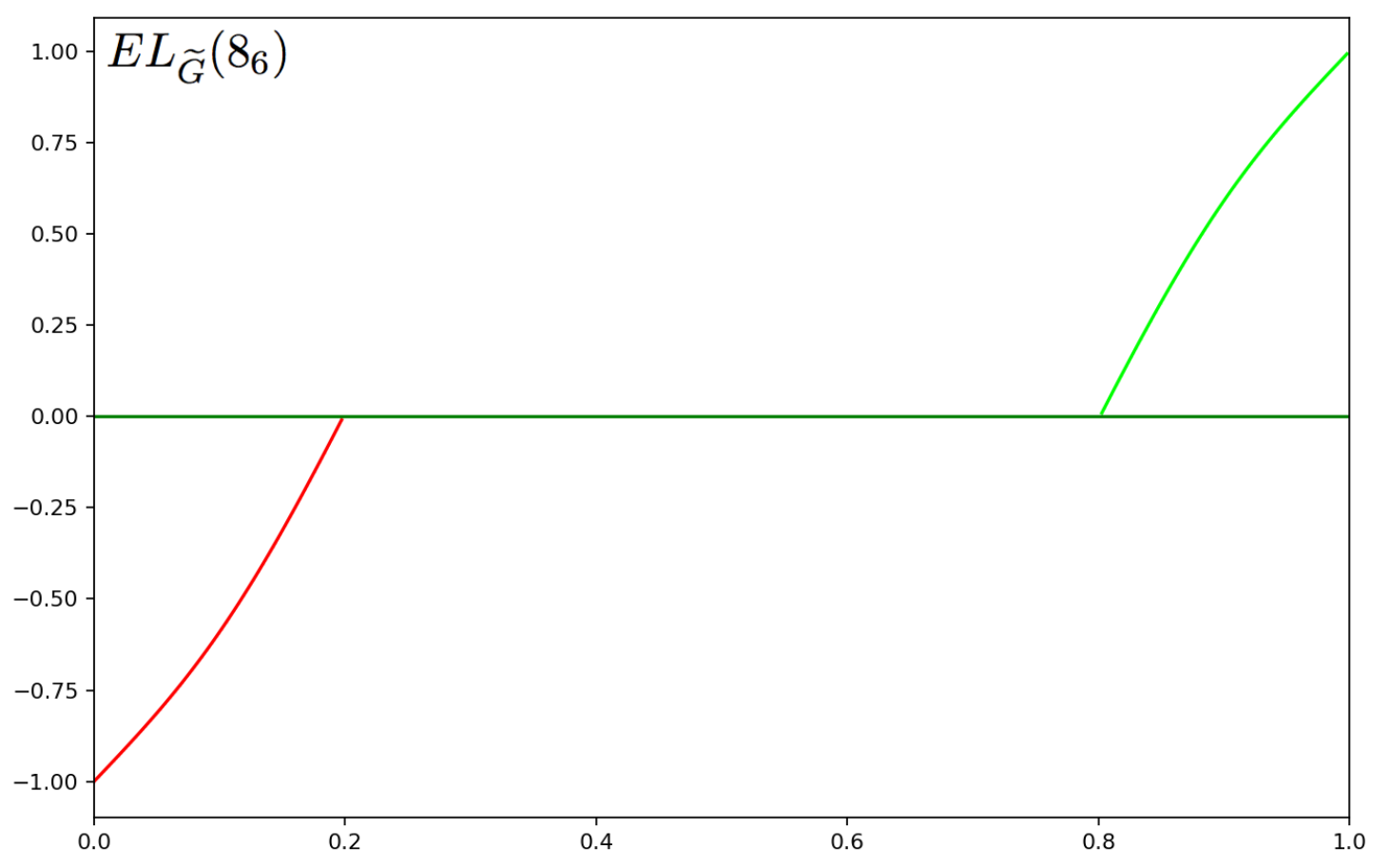}
\caption{Translation Extension Locus $EL_{\widetilde{G}}(8_6)$. }
\label{8_6_E}
\end{figure}

\begin{figure}[H]
\center
\includegraphics[width=93mm]{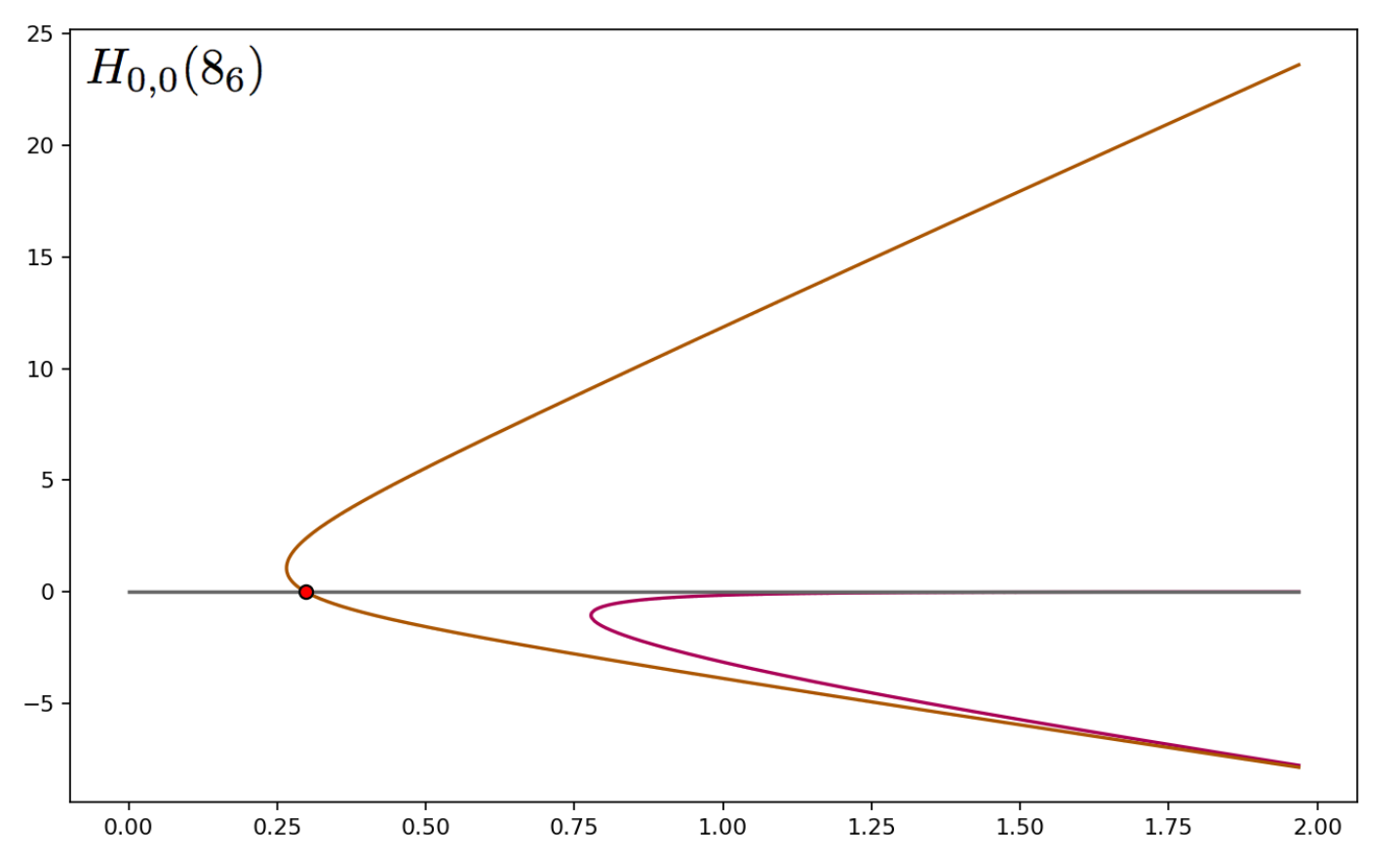}
\caption{Holonomy Extension Locus $HL_{\widetilde{G}}(8_6)$. }
\label{8_6_H}
\fnote{Figure \ref{8_6_E} is the quotient of translation extension locus of $8_6$ under the $\mathbb{Z}/2\mathbb{Z}$ action of integral translation in the $x$ direction.
Figure \ref{8_6_H} is the quotient of the $H_{0,0}(8_6)$ component of the holonomy extension locus of $8_6$ under the action of reflection about the origin. \\
The two-bridge knot $8_6$ is not a double twist knot. Its Alexander polynomial is $2a^4 - 6a^3 + 7a^2 - 6a + 2$, which has a pair (reciprocal to each other) of positive real roots and a pair of unit complex roots.}
\end{figure}

We can see from the above example that the holonomy extension locus and translation extension locus of $8_6$ have similar patterns as those of double twists knots shown in Figure \ref{e94} and Figure \ref{h62}.

When the Alexander polynomial of a two-bridge knot has no positive real roots or unit complex roots, in most cases we are still able to find intervals of left-orderable surgery slopes by constructing $\widetilde{G}$ representations and studying the holonomy extension loci. However, in other cases, it is possible that we may not be able to obtain such intervals of left-orderable surgery slopes using $\widetilde{G}$ representations. The two-bridge knot $6_3$ is such an example. 

The translation extension locus $EL_{\widetilde{G}}(6_3)$ contains only abelian boundary elliptic $\widetilde{G}$ representations, so we look at the holonomy extension locus $HL_{\widetilde{G}}(6_3)$. Let $\widetilde{\rho}$ be a nonabelian boundary hyperbolic $\widetilde{G}$ representation of $\pi_1(6_3)$ in $HL_{\widetilde{G}}(6_3)$. Then the translation number of $\widetilde{\rho}(\mathcal{L})$ must be $\pm 1$. So we only see the $H_{0,\pm 1}(6_3)$ components in the holonomy extension locus of $6_3$, except abelian boundary hyperbolic $\widetilde{G}$ representations which constitute $H_{0,0}(6_3)$.
\begin{figure}[H]
\center
\includegraphics[width=100mm]{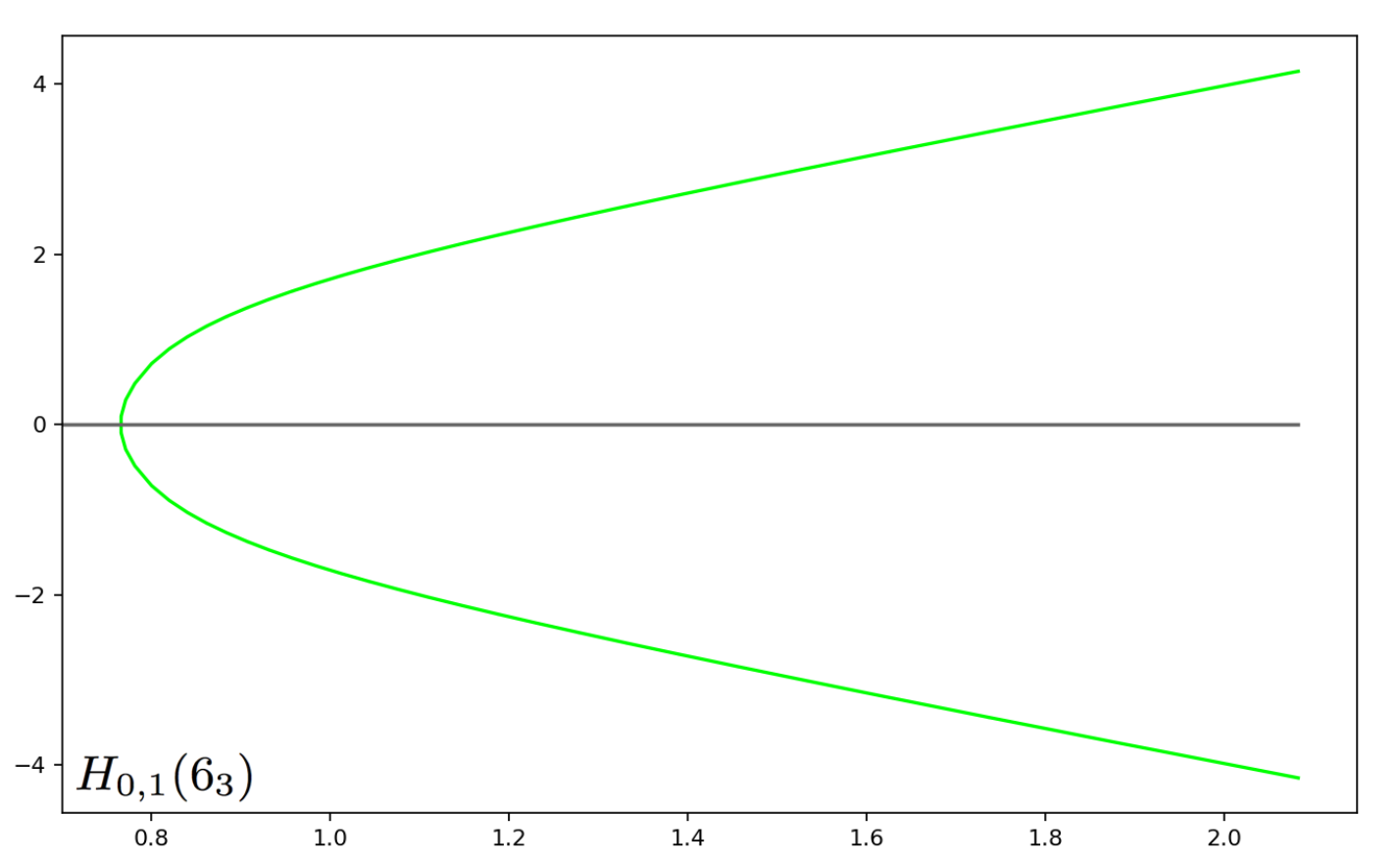}
\caption{Holonomy Extension Locus $HL_{\widetilde{G}}(6_3)$. }
\fnote{ This figure is the $H_{0,1}(6_3)$ component of the holonomy extension locus of $6_3$. (The component $H_{0,-1}(6_3)$ only differ from $H_{0,1}(6_3)$ by reflection about the origin. So we will not show it here.) There are two asymptotes of slope $-2$ and $2$. So for every rational $r\in(-2,2)$, $\pi_1(M(r))$ has a nonabelian $G$ representation. However, such a representation does not lift to $\widetilde{G}$ because the translation number of $\widetilde{\rho}(\mathcal{L})$ does not vanish, which means we cannot obtain an interval of left-orderable surgery slopes by constructing $\widetilde{G}$ representations. }
\end{figure}

\section{Conjecture}\label{conj}

\subsection{Improving the Results}
As observed in Figure \ref{h62}, the arc in the holonomy extension locus $H_{0,0}(6_2)$ has another asymptote of slope $4=-(4m+4)$, in addition to the slope $-8=-4n$ which is confirmed by Proposition \ref{slope}. This phenomenon is also observed in other examples of $J(2m+1,2n)$ with $m<-1$ and $n>0$. Actually, according to our computational data, the other slope should be $-(4m+4)$. Unfortunately the author does not know how prove this. 

But as we will see next, $-4n$ is the boundary slope of some incompressible surface in $S^3-J(2m+1,2n)$.
We follow the notation in \cite{HT_bridge}. 
For rational number $0<\frac{p}{q}<1$, there is a continued fraction expansion
\begin{displaymath}
\frac{p}{q}=[a_1,-a_2,a_3,-a_4,\dots,\pm a_k]=\cfrac{1}{a_1+\cfrac{1}{a_2 + \lastcfrac{1}{a_k} }}, \quad 
\begin{array}{c} 
|a_k|>1\end{array}.
\end{displaymath}
Consider the case when $m<-1$. The two-bridge slope of $J(2m+1,2n)$ is $\displaystyle \cfrac{1}{2m+1+\cfrac{1}{2n}}=[2m+1,-2n]$. Let $n^+$ and $n^-$ be the number of positive and negative numbers in $[2m+1,-2n]$, then $n^+=0$ and $n^-=2$.

Rewrite this slope in the unique continued fraction expansion with each number even, then 
\begin{displaymath}
\cfrac{1}{2m+1+\cfrac{1}{2n}}=\cfrac{1}{2m+2+\cfrac{1-2n}{2n}}=[2m+2,\underbrace{2,\ldots,2}_{2n-1}].
\end{displaymath}
Let $n_0^+$ and $n_0^-$ be the number of positive and negative numbers in $[2m+2,\underbrace{2,\ldots,2}_{2n-1}]$, then $n_0^+=2n-1$ and $n_0^-=1$.

So $m(2,\ldots,2)=2[(n^+-n^-)-(n_0^+-n_0^-)]=-4n$. By \cite[Proposition 2]{HT_bridge}, $-4n$ is the boundary slope of some incompressible surface in $S^3-J(2m+1,2n)$.

\subsection{General Case for Two-Bridge Knots}
In \cite[Lemma 3.6]{gao2}, the author noticed that an arc in the holonomy extension locus approaches the line through the origin of slope equal to the slope of the incompressible surface associated to an ideal point of the character variety. So Dunfield made the following conjecture. 
\begin{conjecture}
Let $K$ be a two-bridge knot. Suppose the Alexander polynomial of $K$ has a root $p_0$.
For any rational $r\in I$ as defined below, Dehn filling of $K$ of slope $r$ is orderable, with
\begin{displaymath}
I= 
\begin{cases} 
(-S_1, -S_2) & \mbox{if } p_0>0, \\ 
(-\infty, k)  & \mbox{if } p_0 \text{ is a unit complex number},
\end{cases}
\end{displaymath}
where $S_1$ and $S_2$ are slopes of incompressible surfaces associated to ideal points of the PSL$_2\mathbb{C}$ character variety of the complement of $K$, and $0<k\leq 2g(K)-1$ is some odd integer.
\end{conjecture}
The interval is optimal in the sense that any interval $I'$ such that $\pi_1(M(r))$ has a nontrivial $\widetilde{PSL_2(\mathbb{R})}$ representation for any rational $r\in I'$, is contained in $I$.
This conjecture is also the motivation of this paper.

\subsection*{Acknowledgements}
The author was supported by US NSF grant DMS-1811156 and Mid-Career Researcher Program (2018R1A2B6004003) through the National Research Foundation funded by the government of Korea.
The author gratefully thanks Nathan Dunfield for introducing her to this problem and providing numerous helpful suggestion. The author would also like to thank Ying Hu for suggesting her read papers by Ryoto Hakamata and Masakazu Teragaito, and thank Masakazu Teragaito, Makoto Sakuma, Sanghyun Kim and the referee for their advice.
Finally, the author would like to thank Marc Culler and Nathan Dunfield for allowing her to use their computer program PE which produces all the graphs in this paper, and teaching her how it works.

\bibliography{two_bridge_reference}
\bibliographystyle{ws-jktr}
\end{document}